\documentclass[runningheads]{llncs}
\usepackage{comment}
\usepackage{bussproofs}
\EnableBpAbbreviations
\usepackage{amsmath,amssymb}
\usepackage{mathtools}
\usepackage{txfonts}
\usepackage{verbatim,cmll}
\usepackage{mathdots}
\usepackage{multicol}
\usepackage{xspace}
\usepackage{color}
\usepackage{xcolor}
\usepackage{graphicx}
\usepackage{cmll}
\usepackage{multirow}
\usepackage{cellspace} 
\def\mc{\multicolumn}
\newcommand{\cupdot}{\mathbin{\mathaccent\cdot\cup}}

\newcommand{\fns}{\footnotesize}

\renewcommand{\emph}{\textbf}

\newcommand{\Prop}{\mathsf{Prop}}

\newcommand{\marginnote}[1]{\marginpar{\raggedright\tiny{#1}}} 
\newcommand{\mpcolor}{\color{red}}
\newcommand{\mpprefix}{MP: }
\newcommand{\mpnote}[1]{{\mpcolor \mpprefix #1 }}
\newcommand{\mpmnote}[1]{\marginnote{\mpnote{#1}}}

\newcommand{\val}[1]{[\![{#1}]\!]}
\newcommand{\descr}[1]{(\![{#1}]\!)}
\renewcommand{\phi}{\varphi}

\newcommand{\wbox}{\ensuremath{\Box}\xspace}
\newcommand{\wdia}{\ensuremath{\Diamond}\xspace}
\newcommand{\bbox}{\ensuremath{\blacksquare}\xspace}
\newcommand{\bdia}{\ensuremath{\Diamondblack}\xspace}

\newcommand{\aor}{\ensuremath{\vee}\xspace}

\newcommand{\vd}{\ \,\textcolor{black}{\vdash}\ \,}

\newcommand{\Rarr}{\Rightarrow}

\newcommand{\RWD}{R_\wdia}
\newcommand{\RBB}{R_\bbox}

\newcommand{\Swdia}{S_{{\!\wdia}}}
\newcommand{\Swbox}{S_{{\!\wbox}}}
\newcommand{\Sbdia}{S_{{\!\bdia}}}

\newcommand{\kentsbox}{\Swbox}
\newcommand{\kentsdiamond}{\Sbdia}
\newcommand{\kbox}{\Box} 
\newcommand{\kdia}{\Diamondblack}

\newcommand{\rulespace}{3.4mm}

\title{Labelled calculi for the logics of rough concepts}
\author{
Ineke van der Berg\inst{1,3}\orcidID{0000-0003-2220-1383}
\and
Andrea De Domenico\inst{1}\orcidID{0000-0002-8973-7011} 
\and
Giuseppe Greco\inst{1}\orcidID{0000-0002-4845-3821} 
\and
Krishna B.~Manoorkar\inst{1}\orcidID{0000-0003-3664-7757}
\and
Alessandra Palmigiano\inst{1,2}\orcidID{0000-0001-9656-7527}
\and
Mattia Panettiere\inst{1}\orcidID{0000-0002-9218-5449}
}
\authorrunning{van der Berg, De Domenico, Greco, Manoorkar, Palmigiano, Panettiere}
\institute{School of Business and Economics, Vrije Universiteit Amsterdam, the Netherlands 
\and
Department of Mathematics and Applied Mathematics, U.~of Johannesburg, South Africa
\and
Department of Mathematical Sciences, Stellenbosch University
\\
\email{\{i.van.der.berg,a.de.domenico,g.greco,k.b.manoorkar,\\a.palmigiano,m.panettiere\}@vu.nl}
}
\date{September 2022}

\begin{document}

\maketitle
\begin{abstract}
 We introduce sound and complete labelled sequent  calculi for the basic normal non-distributive modal logic $\mathbf{L}$ and some of its axiomatic extensions, where the labels  are atomic formulas of the first order language of {\em enriched formal contexts}, i.e., relational structures based on formal contexts which provide complete semantics for these logics. We also extend these calculi to provide a proof system for the logic of {\em rough formal contexts}.
\end{abstract}

\setlength{\abovedisplayskip}{1.7mm}
\setlength{\belowdisplayskip}{1.7mm}
\setlength{\abovedisplayshortskip}{1.7mm}
\setlength{\belowdisplayshortskip}{1.7mm}

\section{Introduction}
In  structural proof theory, powerful solutions  to the problem of introducing analytic calculi for large classes of normal modal logics hinge on incorporating information about the relational semantics of the given logics into the calculi. This  strategy is prominently used in the  design of {\em labelled calculi} \cite{Gabbay1996-GABLDS,negri2005proof,negri2011proof}, a proof-theoretic format using which,  analytic calculi have been introduced for  the axiomatic extensions of the basic normal modal logic defined by modal axioms corresponding to geometric implications in the first order language of Kripke frames.

 Labelled calculi for classical modal logics manipulate sequents $\Gamma\vdash \Delta$ such that  $\Gamma$ and $\Delta$ are multisets of  atomic formulas $xRy$ in the first order language of Kripke frames and labelled formulas $x : A$  interpreted on Kripke frames as $x\Vdash A$, i.e.~as the condition that the modal formula $A$ be satisfied (or forced) at the state $x$ of a given Kripke frame. The labelled calculus $\mathbf{G3K}$ for the basic normal modal logic $\mathbf{K}$ is obtained by expanding the propositional fragment of the Gentzen calculus $\mathbf{G3c}$ with introduction rules for the modal operators   obtained by reading off  the interpretation clauses of $\wbox$- and $\wdia$-formulas on Kripke frames. 
 Labelled calculi for axiomatic extensions of $\mathbf{K}$ defined by Sahlqvist axioms (including the  modal logics T, K4, KB, S4, B, S5) are obtained in \cite{negri2005proof} by augmenting $\mathbf{G3K}$ with  the rules generated by reading off  the first order conditions on Kripke frames corresponding to the given axioms. 
 
 In the present paper, we extend the design principles for the generation of labelled calculi  to  {\em normal   non-distributive modal logics},  a class of normal LE-logics (cf.~\cite{CoPa:non-dist}) the propositional fragment of which coincides with the logic of lattices in which the distributive laws are not necessarily valid.  In \cite{conradie2016categories,conradie2017toward}, non distributive modal logics are used as the underlying environment for an epistemic logic of categories and formal concepts, and in \cite{conradie2021rough} as the logical environment of a  theory unifying Formal Concept Analysis \cite{ganter2012formal} and Rough Set Theory \cite{pawlak1982rough}. 
 
 Specifically, making use of the fact that the basic normal non-distributive modal logic is sound and complete w.r.t.~{\em enriched formal contexts} (i.e., relational structures based on formal contexts from FCA) \cite{conradie2016categories,conradie2017toward},  and that  modal axioms of a certain syntactic shape \cite{CoPa:non-dist} define elementary (i.e.~first order definable) subclasses of enriched formal contexts, we introduce relational labelled calculi for the basic non-distributive modal logic and some of its axiomatic extensions. 
 
 Moreover, we adapt and specialize these calculi for capturing the logic of relational structures of a related type, referred to as {\em rough formal contexts}, which were introduced by Kent in \cite{kent1994rough} as a formal environment for unifying Formal Concept Analysis and Rough Set Theory. In \cite{greco2019logics}, a sound and complete axiomatization for the non-distributive modal logic of rough formal contexts was introduced by circumventing a technical difficulty which in the present paper is shown to be an impossibility, since two of the three first order conditions characterizing rough formal contexts turn out to be {\em not modally definable} in the modal signature which the general theory would associate with them (cf.~Lemma \ref{lemma:reflnotmodallydef}). However, in the richer language of labelled calculi, these first order conditions can still be used to define structural rules  which capture the axiomatization  introduced in \cite{greco2019logics} for the logic of rough formal contexts.
 
 \paragraph{Structure of the paper.} Section \ref{sec:preliminaries} recalls preliminaries on the logic of enriched and rough formal contexts, Section \ref{sec:relcalculus} presents a labelled calculus for the logic of enriched formal contexts and its extensions. Section \ref{sec:kentstory} proves soundness and completeness results for the calculus for the logic of rough formal contexts. We conclude in Section \ref{sec:Conclusions}.
 
\section{Preliminaries}
\label{sec:preliminaries}
In the present section, we recall the definition and relational semantics of the basic normal non-distributive modal logic in the modal signature $\{\wbox, \wdia, {\rhd}\}$ and some of its axiomatic extensions. This logic and similar others have been studied in the context of a research program aimed at introducing the logical foundations of categorization theory \cite{conradie2016categories,conradie2017toward,conradie2021rough}. In this context, $\Box c$ and $\Diamond c$ and ${\rhd} c$ can be given e.g.~the epistemic interpretation of the categories of the objects which are  {\em certainly}, {\em possibly}, and {\em certainly not} members of category $c$, respectively.  Motivated by these ideas, in \cite{conradie2020non}, possible interpretations of (modal) non-distributive logics are systematically discussed also in their connections with their classical interpretation.

\subsection{Basic normal non-distributive modal logic and some of its axiomatic extensions}
\label{ssec:Non-distributive modal logic}

Let $\Prop$ be a (countable or finite) set of atomic propositions. The language $\mathcal{L}$ is defined as follows:
\begin{gather*} 
  \varphi \coloneqq \bot \mid \top \mid p \mid  \varphi \wedge \varphi \mid \varphi \vee \varphi \mid \Box \varphi \mid  \Diamond\varphi\mid {\rhd} \varphi,  
\end{gather*}
where $p\in \Prop$. 
The {\em basic}, or {\em minimal normal} $\mathcal{L}$-{\em logic} is a set $\mathbf{L}$ of sequents $\phi\vdash\psi$,  with $\phi,\psi\in\mathcal{L}$, containing the following axioms:

{{\centering
\begin{tabular}{ccccccccccccc}
     $p \vdash p$ & \quad\quad & $\bot \vdash p$ & \quad\quad & $p \vdash p \vee q$ & \quad\quad & $p \wedge q \vdash p$ & \quad\quad & $\top \vdash \Box\top$ & \quad\quad & $\Box p \wedge \Box q \vdash \Box(p \wedge q)$
     \\
     & \quad & $p \vdash \top$ & \quad & $q \vdash p \vee q$ & \quad & $p \wedge q \vdash q$ &\quad &  $\Diamond\bot \vdash \bot$ & \quad & $\Diamond(p \vee q) \vdash \Diamond p \vee \Diamond q$\\
     & \quad & & \quad &  & \quad &  &\quad &  $\top\vdash {\rhd}\bot$ & \quad & ${\rhd}p \wedge {\rhd}q\vd {\rhd}(p\aor q)$
\end{tabular}
\par}}
\noindent and closed under the following inference rules:
		{\small{
		\begin{gather*}
			\frac{\phi\vdash \chi\quad \chi\vdash \psi}{\phi\vdash \psi}
			\ \ 
			\frac{\phi\vdash \psi}{\phi\left(\chi/p\right)\vdash\psi\left(\chi/p\right)}
			\ \ 
			\frac{\chi\vdash\phi\quad \chi\vdash\psi}{\chi\vdash \phi\wedge\psi}
			\ \ 
			\frac{\phi\vdash\chi\quad \psi\vdash\chi}{\phi\vee\psi\vdash\chi}
\ \ 
			\frac{\phi\vdash\psi}{\Box \phi\vdash \Box \psi}
\ \ 
\frac{\phi\vdash\psi}{\Diamond \phi\vdash \Diamond \psi}
\ \  
\frac{\phi\vdash\psi}{\rhd \psi\vdash \rhd \phi}
		\end{gather*}
		}}
An {\em $\mathcal{L}$-logic} is any extension of $\mathbf{L}$  with $\mathcal{L}$-axioms $\phi\vdash\psi$. In what follows, for any set $\Sigma$ of $\mathcal{L}$-axioms, we let $\mathbf{L}.\Sigma$ denote the axiomatic extension of $\mathbf{L}$ generated by $\Sigma$. Throughout the paper, we will consider all subsets $\Sigma$ of the set of axioms listed in the table below. Some of these axioms are well known from classical  modal logic, and have also cropped up in  \cite{conradie2021rough} in the context of the definition of  relational structures  simultaneously generalizing Formal Concept Analysis and Rough Set Theory. In Proposition \ref{lemma:correspondences}, we list their first-order correspondents w.r.t.~the relational semantics discussed in the next section.

{\small
\begin{center}
\begin{tabular}{rccclclcccl}
\hline
  && $\wdia \wdia A \vdash \wdia A$ && $\wbox A \vd \wbox \wbox A$ &$\quad\quad$&  && $A \vdash \wbox \wdia A$ && $\wdia \wbox A \vd A$ \\
  && $\wbox A \vdash A$ && $A \vd \wdia A$ &&   && $ A \vd {\rhd} {\rhd} A$ &&  \\
\hline
\end{tabular}
\end{center}
 }

\subsection{Relational semantics of $\mathcal{L}$-logics}
\label{ssec:relsem}
The present subsection collects notation,  notions and facts from \cite{conradie2021rough,conradie2020non}.
For any binary relation $T\subseteq U\times V$, and any $U'\subseteq U$  and $V'\subseteq V$, we let $T^c$ denote the set-theoretic complement of $T$ in $U\times V$, and
\begin{equation}\label{eq:def;round brackets}T^{(1)}[U']:=\{v\mid \forall u(u\in U'\Rightarrow uTv) \}  \quad\quad T^{(0)}[V']:=\{u\mid \forall v(v\in V'\Rightarrow uTv) \}.\end{equation}
Well known properties of this construction (cf.~\cite[Sections 7.22-7.29]{davey2002introduction}) are stated in the following lemma.
 \begin{lemma}\label{lemma: basic}
 For any sets $U, V$, $U'$ and $V'$, and for any families of sets $\mathcal{V}$ and $\mathcal{U}$,
~\begin{enumerate}
\item $X_1\subseteq X_2\subseteq U$ implies $T^{(1)}[X_2]\subseteq T^{(1)}[X_1]$, and $Y_1\subseteq Y_2\subseteq V$ implies $T^{(0)}[Y_2]\subseteq T^{(0)}[Y_1]$.
\item $U'\subseteq T^{(0)}[V']$ iff  $V'\subseteq T^{(1)}[U']$.
 \item $U'\subseteq T^{(0)}[T^{(1)}[U']]$ and $V'\subseteq T^{(1)}[T^{(0)}[V']]$.
 \item $T^{(1)}[U'] = T^{(1)}[T^{(0)}[T^{(1)}[U']]]$ and $T^{(0)}[V'] = T^{(0)}[T^{(1)}[T^{(0)}[V']]]$.
 \item $T^{(0)}[\bigcup\mathcal{V}] = \bigcap_{V'\in \mathcal{V}}T^{(0)}[V']$ and $T^{(1)}[\bigcup\mathcal{U}] = \bigcap_{U'\in \mathcal{U}}T^{(1)}[U']$.
\end{enumerate}
 \end{lemma}
 If $R\subseteq U \times V$, and $S \subseteq V \times W$, then  the composition  $R;S \subseteq U \times W$ is defined as follows: 
 \[u (R;S) w \quad  \text{iff} \quad  u \in R^{(0)}[S^{(0)}[w]] \quad \text{iff} \quad \forall v(v S w \Rarr u R v). \]
 
 In what follows, we fix two sets $A$ and $X$, and use $a, b$ (resp.~$x, y$) for elements of $A$ (resp.~$X$), and $B, C, A_j$ (resp.~$Y, W, X_j$) for subsets of $A$ (resp.~of $X$).


A {\em polarity} or {\em formal context} (cf.~\cite{ganter2012formal}) is a tuple $\mathbb{P} =(A,X,I)$, where $A$ and $X$ are sets, and $I \subseteq A \times X$ is a binary relation. In what follows, for any such polarity, we will let $J\subseteq X\times A$ be defined by the equivalence  $xJa$ iff $aIx$.
Intuitively, formal contexts can be understood as abstract representations of databases \cite{ganter2012formal}, so that  $A$ represents a collection of {\em objects}, $X$  a collection of {\em features}, and for any object $a$ and feature $x$, the tuple $(a, x)$ belongs to $I$ exactly when object $a$ has feature $x$.

 As is well known, for every formal context $\mathbb{P} = (A, X, I)$, the pair of maps \[(\cdot)^\uparrow: \mathcal{P}(A)\to \mathcal{P}(X)\quad \mbox{ and } \quad(\cdot)^\downarrow: \mathcal{P}(X)\to \mathcal{P}(A),\]
respectively defined by the assignments $B^\uparrow: = I^{(1)}[B]$ and $Y^\downarrow: = I^{(0)}[Y]$,  form a Galois connection (cf.~Lemma \ref{lemma: basic}.2), and hence induce the closure operators $(\cdot)^{\uparrow\downarrow}$ and $(\cdot)^{\downarrow\uparrow}$ on $\mathcal{P}(A)$ and on $\mathcal{P}(X)$ respectively.\footnote{When $B=\{a\}$ (resp.\ $Y=\{x\}$) we write $a^{\uparrow\downarrow}$ for $\{a\}^{\uparrow\downarrow}$ (resp.~$x^{\downarrow\uparrow}$ for $\{x\}^{\downarrow\uparrow}$).} The fixed points of these closure operators are  referred to as {\em Galois-stable} sets. 
For a formal context $\mathbb{P}=(A,I,X)$, a {\em formal concept} of $\mathbb{P}$ is a tuple $c=(B,Y)$ such that $B\subseteq A$ and $Y\subseteq X$, and $B = Y^\downarrow$ and $Y = B^\uparrow$.  The subset $B$ (resp.~$Y$) is referred to as  the {\em extension} (resp.~the {\em intension}) of $c$ and is denoted by $\val{c}$  (resp.~$\descr{c}$). 
By Lemma \ref{lemma: basic}.3, the sets $B$ and $Y$ are  Galois-stable. 
It is well known (cf.~\cite{ganter2012formal}) that the set  of formal concepts of a formal context $\mathbb{P}$,  with the order defined by
\smallskip

{{\centering
$c_1 \leq c_2 \quad \text{iff} \quad \val{c_1} \subseteq \val{c_2} \quad \text{iff} \quad \descr{c_2} \subseteq \descr{c_1}$,
\par}}
\smallskip

\noindent forms a complete lattice, namely  the {\em concept lattice} of $\mathbb{P}$, which we denote by $\mathbb{P}^+$.



For the language $\mathcal{L}$ defined in the previous section, an {\em enriched formal $\mathcal{L}$-context} is a tuple $\mathbb{F} =(\mathbb{P}, R_\wbox, R_\Diamond, R_{\rhd})$, where $R_\wbox \subseteq A \times X$ and $R_\wdia \subseteq X \times A$ and $R_{\rhd}\subseteq A\times A$ are {\em $I$-compatible} relations, that is, for all $a,b \in A$, and all $x \in X$, the sets $R_\wbox^{(0)}[x]$, $R_\wbox^{(1)}[a]$, $R_\wdia^{(0)}[a]$, $R_\wdia^{(1)}[x]$,
$R_\rhd^{(0)}[b]$, $R_\rhd^{(1)}[a]$  are Galois-stable in $\mathbb{P}$. As usual in modal logic, these relations can be interpreted in different ways, for instance as the epistemic attributions of features to objects by agents. 


A {\em valuation} on such an $\mathbb{F}$ 
is a map $V\colon\Prop\to \mathbb{P}^+$. For every  $p\in \Prop$, we let  $\val{p}: = \val{V(p)}$ (resp.~$\descr{p}: = \descr{V(p)}$) denote the extension (resp.~the intension) of the interpretation of $p$ under $V$.  
A {\em model} is a tuple $\mathbb{M} = (\mathbb{F}, V)$ where $\mathbb{F} = (\mathbb{P}, R_{\Box}, R_{\Diamond}, R_{\rhd})$ is an enriched formal context and $V$ is a  valuation on $\mathbb{F}$.  For every $\phi\in \mathcal{L}$, the following `forcing' relations can be recursively defined as follows: 
\smallskip


\smallskip

{{\centering 
\small
\begin{tabular}{l@{\hspace{1em}}l@{\hspace{2em}}l@{\hspace{1em}}l}
$\mathbb{M}, a \Vdash p$ & iff $a\in \val{p}_{\mathbb{M}}$ &
$\mathbb{M}, x \succ p$ & iff $x\in \descr{p}_{\mathbb{M}}$ \\
$\mathbb{M}, a \Vdash\top$ & always &
$\mathbb{M}, x \succ \top$ & iff   $a I x$ for all $a\in A$\\
$\mathbb{M}, x \succ  \bot$ & always &
$\mathbb{M}, a \Vdash \bot $ & iff $a I x$ for all $x\in X$\\
$\mathbb{M}, a \Vdash \phi\wedge \psi$ & iff $\mathbb{M}, a \Vdash \phi$ and $\mathbb{M}, a \Vdash  \psi$ & 
$\mathbb{M}, x \succ \phi\wedge \psi$ & iff $(\forall a\in A)$ $(\mathbb{M}, a \Vdash \phi\wedge \psi \Rightarrow a I x)$
\\
$\mathbb{M}, x \succ \phi\vee \psi$ & iff  $\mathbb{M}, x \succ \phi$ and $\mathbb{M}, x \succ  \psi$ & 
$\mathbb{M}, a \Vdash \phi\vee \psi$ & iff $(\forall x\in X)$ $(\mathbb{M}, x \succ \phi\vee \psi \Rightarrow a I x)$.
\end{tabular}
\par}}
\smallskip

\noindent As to the interpretation of modal formulas:
\smallskip

{{\centering
\small
\begin{tabular}{llcll}
$\mathbb{M}, a \Vdash \Box\phi$ &  iff $(\forall x\in X)(\mathbb{M}, x \succ \phi \Rightarrow a R_\Box x)$ & \quad\quad &
$\mathbb{M}, x \succ \Box\phi$ &  iff $(\forall a\in A)(\mathbb{M}, a \Vdash \Box\phi \Rightarrow a I x)$\\
$\mathbb{M}, x \succ \Diamond\phi$ &  iff for all $ a\in A$, if $\mathbb{M}, a \Vdash \phi$ then $x R_\Diamond a$ &&
$\mathbb{M}, a \Vdash \Diamond\phi$ & iff $(\forall x\in X)(\mathbb{M}, x \succ \Diamond\phi \Rightarrow a I x)$  \\
 $\mathbb{M}, a \Vdash {\rhd}\phi$ &  iff $(\forall b\in A)(\mathbb{M}, b \Vdash \phi \Rightarrow a R_{\rhd} b)$ &  &
 $\mathbb{M}, x \succ {\rhd}\phi$ &  iff $(\forall a\in A)(\mathbb{M}, a \Vdash {\rhd}\phi  \Rightarrow a I x)$.\\
\end{tabular}
\par}}
\smallskip

\noindent The definition above ensures that, for any $\mathcal{L}$-formula $\varphi$,\smallskip

{{\small\centering
$\mathbb{M}, a \Vdash \phi$  iff  $a\in \val{\phi}_{\mathbb{M}}$, \quad  and \quad$\mathbb{M},x \succ \phi$  iff  $x\in \descr{\phi}_{\mathbb{M}}$. \par}}

\smallskip
\noindent Finally, as to the interpretation of sequents:
\smallskip

{{\small\centering
$\mathbb{M}\models \phi\vdash \psi$ \quad iff \quad $\val{\phi}_{\mathbb{M}}\subseteq \val{\psi}_{\mathbb{M}}$\quad  iff  \quad  $\descr{\psi}_{\mathbb{M}}\subseteq \descr{\phi}_{\mathbb{M}}$. 
\par}}
\smallskip

A sequent $\phi\vdash \psi$ is {\em valid} on an enriched formal context $\mathbb{F}$ (in symbols: $\mathbb{F}\models \phi\vdash \psi$) if $\mathbb{M}\models \phi\vdash \psi$  for every model $\mathbb{M}$ based on $\mathbb{F}$. 
The basic non-distributive logic $\mathbf{L}$ is sound and complete w.r.t.~the class of enriched formal contexts (cf.~\cite{conradie2021rough}). 

Then, via a general canonicity result (cf.~\cite{CoPa:non-dist}), the following proposition (cf.~\cite[Proposition 4.3]{conradie2021rough}) implies that, for any subset $\Sigma$ of the set of axioms at the end of Section \ref{ssec:Non-distributive modal logic}, the logic $\mathbf{L}.\Sigma$ is complete w.r.t.~the class of enriched formal contexts defined by those first-order sentences in the statement of the proposition below  corresponding to the axioms in $\Sigma$. 

These first order sentences are compactly represented as inclusions of relations defined as follows. For any enriched formal context $\mathbb{F} = (\mathbb{P}, R_{\Box}, R_{\Diamond}, R_{\rhd} )$, 
the relations $R_{\Diamondblack}\subseteq X\times A$,  $R_\blacksquare\subseteq A\times X$ and $R_{\blacktriangleright}\subseteq A\times A$ are defined by $xR_{\Diamondblack} a$ iff $aR_{\wbox} x$, and $a R_\blacksquare x$ iff $x R_\wdia a$, and $a R_{\blacktriangleright}b$ iff $bR_{\rhd}a$. Moreover, for all relations $R, S\subseteq A\times X$ we let $R; S\subseteq A\times X$ be defined\footnote{These compositions   and those defined in Section \ref{ssec:relsem} are pairwise different, since each of them involves different   types of relations. However, the types of the relations involved in each definition provides a unique reading of such compositions, which justifies our abuse of notation.} by $a(R; S)x$ iff $a\in R^{(0)}[I^{(1)}[S^{(0)}[x]]]$, and for all relations $R, S\subseteq X\times A$ we let $R; S\subseteq X\times A$ be defined by $x(R; S)a$ iff $x\in R^{(0)}[I^{(0)}[S^{(0)}[a]]]$.

\begin{proposition}
\label{lemma:correspondences}
For any enriched formal context $\mathbb{F} = (\mathbb{P}, R_\Box, R_\Diamond, R_{\rhd})$:
\smallskip

{\small
{{\centering 
\begin{tabular}{rlcl c rlcl}
$1$. & $\mathbb{F}\models \Box\phi\vdash \phi$ & iff & $ R_\Box\subseteq I$. & &
$5$. & $\mathbb{F}\models \Diamond\Diamond\phi\vdash \Diamond\phi $ & iff & $ R_{\Diamond}\subseteq R_{\Diamond}\, ; R_{\Diamond}$. \\
$2$. & $\mathbb{F}\models \phi\vdash \Diamond\phi $ & iff & $R_\wdia\subseteq J$. &&
$6$. & $\mathbb{F}\models \phi\vdash \Box\Diamond\phi $ & iff & $\quad R_{\Diamond} \subseteq R_{\Diamondblack}$.\\
$3$. & $\mathbb{F}\models \Box\phi\vdash \Box\Box\phi $ & iff & $ R_{\Box}\subseteq R_{\Box}\, ; R_{\Box}$. & &
$7$.&  $\mathbb{F}\models \wdia\wbox\phi\vdash \phi $ & iff &  $ \quad  R_{\Diamondblack}  \subseteq R_{\Diamond}$. \\
$4$.&  $\mathbb{F}\models \phi\vdash {\rhd}{\rhd}\phi $ & iff & $ R_{\rhd} = R_{\blacktriangleright}$. \\

\end{tabular}
\par}}
 }
\end{proposition}

\noindent The proposition above motivated the introduction of the notion of conceptual approximation space in \cite{conradie2021rough}, as a subclass of the enriched formal contexts modelling the ${\rhd}$-free fragment of the language $\mathcal{L}$. 
%
%
%
%
A {\em conceptual approximation space} is an enriched formal context $\mathbb{F} = (\mathbb{P}, R_\Box, R_\Diamond)$ verifying the first order sentence
$R_{\Box};R_{\blacksquare} \subseteq I$.
Such an $\mathbb{F}$ is {\em reflexive} if $R_\Box\subseteq I$ and $R_\wdia\subseteq J$, is {\em symmetric} if $R_{\Diamond}  = R_{\Diamondblack}$ or equivalently if $R_{\blacksquare}  = R_{\Box}$, and  is {\em transitive} if $R_{\Box}\subseteq R_{\Box}\, ; R_{\Box}$ and $R_{\Diamond}\subseteq R_{\Diamond}\, ; R_{\Diamond}$ (cf.~\cite{conradie2021rough,conradie2022modal} for a discussion on terminology).

\subsection{The logic of rough formal contexts}
\label{ssec:kentstructures}

Examples of conceptual approximation spaces have cropped up in the context of  Kent's proposal for a simultaneous generalization of approximation spaces from RST and formal contexts from FCA  \cite{kent1996rough}. Specifically, Kent introduced {\em rough formal contexts} 
as tuples $\mathbb{G} = (\mathbb{P}, E)$ such that $\mathbb{P} = (A, X, I)$ is a polarity, and $E\subseteq A\times A$ is an equivalence relation.   
The relation $E$ induces two  relations $R_\wbox, \kentsbox\subseteq A\times X$  defined as follows: for every $a\in A$ and $x\in X$,
\begin{equation}\label{eq:lax approx}
aR_\wbox x \, \mbox{ iff }\, \exists b(aEb \, \&\, bIx) 
\quad\quad\quad\quad
a\kentsbox x \, \mbox{ iff }\,  \forall b(aEb\Rightarrow bIx) 
\end{equation}
The reflexivity of $E$ implies that $\kentsbox\subseteq I\subseteq R_\wbox$; hence, 
  $R_\wbox$ and $\kentsbox$ can  respectively be regarded as the {\em lax}, or {\em upper}, and as the {\em strict}, or {\em lower}, approximation of $I$ relative to $E$. 
 For any rough formal context $\mathbb{G} = (\mathbb{P}, E)$, let $\kentsdiamond\subseteq X\times A$ be defined by the equivalence $x\kentsdiamond a$ iff $a\kentsbox x$, 
\begin{lemma}
\label{lemma:s_as_i_e}
If $\mathbb{G} = (\mathbb{P}, E)$ is a rough formal context, then $\kentsdiamond = J;E$. 
\end{lemma}
\begin{proof}
For any $a\in A$ and $x\in X$,

{\small
{{\centering 
\begin{tabular}{rcll}
$x\kentsdiamond a$ & iff & $a\kentsbox x$&\quad Definition of $\kentsdiamond$\\ 
&iff & $\forall b(bEa \Rightarrow bIx)$ &\quad Definition of $\kentsbox$\\ 
&iff & $\forall b(bEa \Rightarrow xJb)$ &\quad Definition of $J$\\ 
&iff & $E^{(0)}[a] \subseteq J^{(1)}[x]$ &\quad notation $T^{(0)}[-]$ and $T^{(1)}[-]$\\ 
&iff & $x\in J^{(0)}[E^{(0)}[a]]$ &\quad Lemma \ref{lemma: basic}.2\\ 
  &iff   & $x (J ; E) a$. &\quad Definition of $J;E$ \\
\end{tabular}
\par}}
 }
\end{proof}
In \cite[Section 5]{conradie2021rough} and \cite[Section 3]{greco2019logics}, the logic of rough formal contexts  was introduced, based on the theory of enriched formal contexts as models of non-distributive modal logics, the characterization results reported on in Proposition  \ref{lemma:correspondences}, and the following:
\begin{lemma} \label{thm:equivalent_s_rewriting} (\cite[Lemma 5.3]{conradie2021rough})
For any polarity $\mathbb{P} = (A, X, I)$, and any $I$-compatible relation $E\subseteq A\times A$ such that its associated $\kentsbox\subseteq A\times X$ (defined as in \eqref{eq:lax approx}) is $I$-compatible,\footnote{Notice that $E$ being $I$-compatible does not imply that $\kentsbox$ is. Let $\mathbb{G} = (\mathbb{P}, E)$ s.t.~ $A: = \{a, b\}$, $X: = \{x, y\}$, $I: = \{(a, x), (a, y), (b, y)\}$, and $E: = A\times A$. Then $E$ is $I$-compatible. However,  $\kentsbox = \{(a, y), (b, y)\}$ is not, as $\kentsbox^{(0)}[x] = \varnothing$ is not Galois stable, since $\varnothing^{\uparrow\downarrow} = X^{\downarrow} = \{a\}$. In \cite{greco2019logics}, it was remarked that $\kentsbox$ being $I$-compatible does not imply that $E$ is.}

{{\centering 
$E$ is reflexive \quad iff \quad $\kentsbox \subseteq I$;
\quad\quad\quad and \quad\quad\quad
$E$ is transitive \quad iff \quad $\kentsbox \subseteq \kentsbox ; \kentsbox$.
\par}}
\end{lemma}
These results imply that the characterizing properties of rough formal contexts can be taken as completely axiomatised in the modal language $\mathcal{L}$ via the following axioms: 
\smallskip

{{\centering
$\wbox \phi\vd \phi$ $\quad \quad \wbox \phi\vd \wbox\wbox \phi$  $\quad\quad \phi\vd {\rhd}{\rhd}\phi$.
\par}}
\smallskip

\noindent Clearly, any rough formal context $\mathbb{G} = (\mathbb{P}, E)$ such that $E$ is $I$-compatible is an enriched formal $\mathcal{L}_{\rhd}$-context, where $\mathcal{L}_{\rhd}$ is the $\{\wbox, \wdia\}$-free fragment of $\mathcal{L}$. However, interestingly, it is impossible to capture the reflexivity and transitivity of $E$ by means of $\mathcal{L}_{\rhd}$-axioms, as the next lemma shows:

\begin{lemma}
\label{lemma:reflnotmodallydef}
The  class of enriched formal $\mathcal{L}_{\rhd}$-contexts $\mathbb{F} = (\mathbb{P}, R_\rhd)$ such that $R_\rhd\subseteq A\times A$ is reflexive (resp.~transitive) is not modally definable in its associated language $\mathcal{L}_{\rhd}$.
\end{lemma}
\begin{proof}
Assume for contradiction that $\mathcal{L}_{\rhd}$-axioms $\phi\vd \psi$ and $\chi\vd \xi$ exist such that $\mathbb{F}\models \phi\vd \psi$ iff $R_\rhd$ is reflexive, and $\mathbb{F}\models \chi\vd \xi$ iff $R_\rhd$ is transitive for any enriched formal $\mathcal{L}_{\rhd}$-context $\mathbb{F}= (\mathbb{P}, R_\rhd)$. Then, these equivalences would hold in particular for those special formal $\mathcal{L}_{\rhd}$-contexts $\mathbb{F} = (\mathbb{P}_W, R_\rhd)$ such that $\mathbb{P}_W = (W_A, W_X, I_{\Delta^c})$ such that $W_A = W_X = W$ for some set $W$, and $aI_{\Delta^c} x$ iff $a\neq x$, and $R_{\rhd}: = H_{R^c}$   is defined as $a H_{R^c}b$ iff $(a, b)\notin R$ for some binary relation $R\subseteq W\times W$. By construction, letting $\mathbb{X} = (W, R)$, the following chain of equivalences holds:
$\mathbb{F}\models \phi\vd \psi$ iff $\val{\phi}_V\subseteq \val{\psi}_V$ for every valuation $V: \Prop\to \mathbb{P}^+$. However, by construction, $\mathbb{P}^+\cong \mathcal{P}(W)$ (cf.~\cite[Proposition 3.4]{conradie2021rough}). Moreover, the definition of the forcing relation $\Vdash$ on $\mathbb{F}$ implies that 
\begin{center}
    \begin{tabular}{cll}
         $\val{{\rhd}\phi}$ $=$ $R_{\rhd}^{(0)}[\val{\phi}]$ $=$ $H_{R^c}^{(0)}[\val{\phi}]$ 
         & $=$ & $ \{b\in W_A\mid \forall a(a\Vdash \phi\Rightarrow aR^cb)\}$\\ 
          &  $=$ & $ \{b\in W_A\mid \forall a(aRb \Rightarrow a \not \Vdash \phi)\} $\\

    \end{tabular}
\end{center}
That is, restricted to the class of $\mathcal{L}_{\rhd}$-contexts which arise from classical Kripke frames $\mathbb{X} = (W, R)$ in the way indicated above, the interpretation of ${\rhd}$-formulas coincides with the interpretation of $\wbox\neg$-formulas in the language of classical modal logic, which induces a translation $\tau$, from $\mathcal{L}_{\rhd}$-formulas to formulas in the language of classical modal logic, which is preserved and reflected from the  special formal $\mathcal{L}_{\rhd}$-contexts $\mathbb{F}$ to the Kripke frames with which they are associated. Therefore, by construction, for any Kripke frame $\mathbb{X} = (X, R)$,  $R$ is irreflexive iff $H_{R^c}$ is reflexive iff 
$\mathbb{F}\models\phi\vd\psi$ iff $\mathbb{X}\models \tau(\phi)\vdash \tau(\psi)$, contradicting the well known fact that the class of Kripke frames $\mathbb{X} = (X, R)$ such that $R$ is irreflexive  is not modally definable.

Reasoning similarly,  to show the statement concerning transitivity, it is enough to see that the class  of Kripke frames $\mathbb{X} = (W, R)$ s.t.~$R^c$ is transitive is not modally definable. Consider the Kripke frames $\mathbb{X}_i= (W_i, R_i)$ such that $W_i= \{a_i, b_i \}$, $R_i = \{(a_i,b_i) \}$, for $1\leq i\leq 2$. Clearly,  $R_i^c$  is transitive in $\mathcal{F}_i$, so the two frames satisfy the property. However, their disjoint union  $\mathbb{X}_1 \cupdot \mathbb{X}_2 =(W,R)$, given by $W= \{a_1,b_1,a_2,b_2\}$ and $R =\{(a_1,b_1), (a_2,b_2)\}$, does not: indeed, $(a_1, a_2), (a_2, b_1)\in R^c$ but $(a_1,b_1)\notin R^c$.   Hence, the statement follows from the Goldblatt-Thomason theorem for classical modal logic. 

\end{proof}

\section{Relational labelled calculi for $\mathcal{L}$-logics}
\label{sec:relcalculus}

Below,  $p, q$ denote atomic propositions; $a, b, c$ (resp.~$x, y, z$) are labels corresponding to objects (resp.~features). Given labels $a$, $x$ and a modal formula $A$, well-formed formulas are of the type $a: A$ and $x:: A$, while $\varphi, \psi$ are meta-variables for well-formed formulas. Well-formed  terms are of any of the following shapes: $aIx$, $aR_\wbox x$, $x R_\wdia a$, $aR_\bbox x$, $x R_\bdia a$, and $t_1 \Rarr t_2$, where $t_1$ is of any of the following shapes:  $a R_\wbox x$, $a R_\bbox x$, $y R_\wdia a$, $y R_\bdia a$, $a R_\rhd b$, $a R_\blacktriangleright b$,  and $t_2$ is of the form $a I y$. Relational terms $t_1 \Rarr t_2$ are interpreted as $\forall u (t_1 \rightarrow t_2)$ where $u$ is the variable shared by $t_1$ and $t_2$. 
A sequent is an expression of the form $\Gamma \vd \Delta$, where $\Gamma, \Delta$ are meta-variables for multisets of well-formed formulas or terms. For any labels $u$, $v$ and relations $R$, $S$ we write $u (R;S) v$ as a shorthand for the term $w S v \Rarr u R w$. 


\subsection{Labelled calculus $\mathbf{R.L}$ for the basic $\mathcal{L}$-logic}
\label{sec:calculus}
{{\centering
\begin{tabular}{c}
\mc{1}{c}{\rule[-1.85mm]{0mm}{8mm}\textbf{Initial rules and cut rules}\rule[-1.85mm]{0mm}{6mm}} \\
\AXC{$\ $}
\LL{Id$_{\,a:p}$}
\UIC{$\Gamma, a: p \vd a: p, \Delta$}
\DP
 \   
\AXC{$\ $}
\RL{Id$_{\,x::p}$}
\UIC{$\Gamma, x:: p \vd x:: p, \Delta$}
\DP
 \ \  
\AXC{$ \ $}
\LL{$\bot$}
\UIC{$\Gamma \vd x:: \bot, \Delta$}
\DP 
 \   
\AXC{$ \ $}
\RL{$\top$}
\UIC{$\Gamma \vd a: \top, \Delta$}
\DP
\\[\rulespace]
\AXC{$\Gamma \vd a: A, \Delta$}
\AXC{$\Gamma', a: A \vd \Delta'$}
\LL{Cut$_{\,aa}$}
\BIC{$\Gamma, \Gamma' \vd \Delta, \Delta'$}
\DP
 \  \ 
\AXC{$\Gamma \vd x:: A, \Delta$}
\AXC{$\Gamma', x:: A \vd \Delta'$}
\RL{Cut$_{\,xx}$}
\BIC{$\Gamma, \Gamma' \vd  \Delta, \Delta'$}
\DP \\[\rulespace]
\end{tabular}
\par}}

{{\centering
\begin{tabular}{rl}
\mc{2}{c}{\rule[-1.85mm]{0mm}{8mm}\textbf{\ \ \ \ Switch rules$^\ast$} \rule[-1.85mm]{0mm}{6mm}}\\
\AXC{$\Gamma,  x:: B \vd x:: A, \Delta$}
\LL{S$xa$}
\UIC{$\Gamma,  a: A \vd  a: B, \Delta $}
\DP
 \ & \  
\AXC{$\Gamma,  a: A \vd  a: B, \Delta $}
\RL{S$ax$}
\UIC{$\Gamma,  x:: B \vd x:: A, \Delta$}
\DP
 \\[\rulespace]
\ \ \ \ 
\AXC{$ \Gamma, y R_\wdia a \Rarr b I y   \vd   b:A , \Delta $}
\LL{S$a \wdia x$}
\UIC{$\Gamma, x::A \vd x R_\Diamond a, \Delta$}
\DP
 \ & \ 
\AXC{$ \Gamma,a: A  \vd   a R_\wbox x, \Delta $}
\RL{S$ a \wbox x$}
\UIC{$\Gamma,  b R_\wbox x \Rarr b I y  \vd  y::A, \Delta$}
\DP
 \\[\rulespace]
\AXC{$\Gamma,  b R_\wbox x \Rarr b I y  \vd  y::A, \Delta$}
\LL{S$ x\wbox a$}
\UIC{$ \Gamma,a: A \vd   a R_\wbox x, \Delta $}
\DP
\ & \ 
\AXC{$\Gamma, x::A \vd x R_\Diamond a, \Delta$}
\RL{S$ x \wdia a$}
\UIC{$ \Gamma, y R_\wdia a \Rarr b I y   \vd   b:A , \Delta $}
\DP
 \\[\rulespace]
\AXC{$ \Gamma, b:A  \vd    y R_\wdia a \Rarr b I y , \Delta $}
\LL{S$ a \wdia x$}
\UIC{$\Gamma, x R_\Diamond a \vd  x::A, \Delta$}
\DP
 \ & \ 
\AXC{$ \Gamma,a: A  \vd   a R_\wbox x, \Delta $}
\RL{S$ a \wbox x $}
\UIC{$\Gamma, y:: A \vd b R_\wbox x \Rarr b I y, \Delta$}
\DP
 \\[\rulespace]
\AXC{$\Gamma,  y::A \vd   b R_\wbox x \Rarr b I y, \Delta$}
\LL{S$ x\wbox a $}
\UIC{$ \Gamma,a R_\wbox x\vd  a: A   , \Delta $}
\DP
\ & \ 
\AXC{$\Gamma, x R_\Diamond a \vd  x::A , \Delta$}
\RL{S$ x \wdia a $}
\UIC{$ \Gamma,b:A    \vd   y R_\wdia a \Rarr b I y,  \Delta $}
\DP
 \\[\rulespace]
\end{tabular}
\par}}

{{\centering
\begin{tabular}{rl}
\ \ \ \ \ 
\AXC{$\Gamma,     b R_\rhd a \Rarr b I y \vd y::A,  \Delta$}
\LL{S$ x{\rhd} a$}
\UIC{$ \Gamma,  c: A   \vd  c R_\rhd a, \Delta $}
\DP
\ & \ 
\AXC{$ \Gamma,c:A    \vd   c R_\rhd a,  \Delta $}
\RL{S$ a {\rhd} x $}
\UIC{$\Gamma,     b R_\rhd a \Rarr b I y \vd y::A,  \Delta$}
\DP
 \\[\rulespace]
\AXC{$\Gamma,   y::A    \vd b R_\rhd a \Rarr b I y, \Delta$}
\LL{S$ x{\rhd} a$}
\UIC{$ \Gamma, c R_\rhd a   \vd   c: A , \Delta $}
\DP
\ & \ 
\AXC{$ \Gamma, c R_\rhd a \vd  c:A   ,  \Delta $}
\RL{S$ a {\rhd} x $}
\UIC{$\Gamma, y::A    \vd  b R_\rhd a \Rarr b I y,  \Delta$}
\DP
 \\[\rulespace]
\mc{2}{p{0.9\textwidth}}{\centering ${}^\ast$Side condition: the variables  $x,y,a,b$  occurring as labels of a formula \\ in the premise of any of these rules must not occur in $\Gamma, \Delta$.}
\\
\end{tabular}
\par}}

\noindent Switch rules  for  $R_\bbox$,  $R_\bdia$, and $R_\blacktriangleright$ are analogous to those for $R_\wbox$,  $R_\wdia$, and $R_\rhd$. These rules encode the I-compatibility conditions of $R_\bbox$,  $R_\bdia$,$R_\blacktriangleright$,$R_\wbox$,  $R_\wdia$, and $R_\rhd$ (cf.\ Remark \ref{remark:switchicomp}).

{{\centering
\begin{tabular}{rl}
\mc{2}{c}{\rule[-1.85mm]{0mm}{8mm}\textbf{Approximation rules$^*$} \rule[-1.85mm]{0mm}{6mm}}\\
\ \ \ \ \ \ \ \ \ \ \ \AXC{$\Gamma,  x:: A \vd aIx, \Delta$}
\LL{$approx_x$}
\UIC{$ \Gamma \vd a: A, \Delta $}
\DP
\ & \ 
  \AXC{$\Gamma,  a: A \vd aIx, \Delta$}
\RL{$approx_a$}
\UIC{$ \Gamma \vd   x::A, \Delta $}
\DP
 \\[\rulespace]
\mc{2}{p{0.9\textwidth}}{\centering${}^\ast$Side condition: the variables  $x,y$  occurring as labels of a formula \\ in the premise of any of these rules must not occur in $\Gamma, \Delta$.}
\\
\end{tabular}
\par}}

\noindent For  $T, T'\in \{R_\wdia,\, J,\, J;I,\, J;R_\wbox,\, J;R_\rhd,\, R_\bdia,\, J;R_\bbox,\, J;R_\blacktriangleright \}$ and  $S, S'\in \{R_\wbox,\, I,\, I;J,\, I;R_\wdia,\, I;R_\bdia, \, R_\bbox\}$ and for all labels  $u,v, w$ of the form $a$ or $x$, we have the following switch rules:

{{\centering
\begin{tabular}{rl}
\mc{2}{c}{\rule[-1.85mm]{0mm}{8mm}\textbf{Pure structure switch rules$^*$} \rule[-1.85mm]{0mm}{6mm}}\\
\AXC{$\Gamma,  x T u    \vd x T' v, \Delta$}
\LL{S$(I;S)$}
\UIC{$ \Gamma, a (I;T') v  \vd   a (I;T) u, \Delta $}
\DP
\ & \ 
\AXC{$\Gamma,  a S u    \vd a S' v, \Delta$}
\RL{S$(J;T)$}
\UIC{$ \Gamma, x (J;S') v  \vd   x (J;S) u, \Delta $}
\DP
 \\[\rulespace]
 \AXC{$ \Gamma, a (I;T') v  \vd   a (I;T) u, \Delta $}
\LL{-S$(I;S)$}
\UIC{$\Gamma,  x T u    \vd x T' v, \Delta$}
\DP
\ & \ 
\AXC{$ \Gamma, x (J;S') v  \vd   x (J;S) u, \Delta $}
\RL{-S$(J;T)$}
\UIC{$\Gamma,  a S u    \vd a S' v, \Delta$}
\DP
 \\[\rulespace]
 \AXC{$\Gamma   \vd  a S u, \Delta$}
\LL{Id$(I;J)_R$}
\UIC{$ \Gamma    \vd   a (I;(J;S)) u , \Delta $}
\DP
\ & \ 
\AXC{$\Gamma   \vd  x T u, \Delta$}
\RL{Id$(J;I)_R$}
\UIC{$ \Gamma    \vd   x (J;(I;T)) u , \Delta $}
\DP
 \\[\rulespace]
 \AXC{$\Gamma,  a S u \vd  \Delta$}
\LL{Id$(I;J)_L$}
\UIC{$ \Gamma, a (I;(J;S)) u    \vd    \Delta $}
\DP
\ & \ 
\AXC{$\Gamma,  x T u  \vd  \Delta$}
\RL{Id$(J;I)_L$}
\UIC{$ \Gamma, x (J;(I;T)) u     \vd   \Delta $}
\DP
 \\[\rulespace]
\mc{2}{p{0.9\textwidth}}{\centering${}^\ast$Side condition: the variable  $x$  (resp.~$a$) occurring  in the premise of  rules S$(I;S)$, -S$(I;S)$ (resp.~S$(J;T)$, -S$(J;T)$ ) must not occur in $\Gamma, \Delta$.}
\\
\end{tabular}
\par}}

\noindent The rules above  encode the definition of  $I$-composition of relations on formal contexts \cite[Definition3.10]{conradie2021rough}.

{{\centering
\begin{tabular}{c}
\mc{1}{c}{\rule[-1.85mm]{0mm}{8mm} \textbf{Adjunction rules}\rule[-1.85mm]{0mm}{8mm}} \\

\AXC{$\Gamma \vd  x R_\wdia a,  \Delta$}
\LL{\fns $\wdia \dashv \bbox$}
\UIC{$\Gamma \vd a  R_\bbox x,  \Delta$}
\DP 
 \ \ \  
\AXC{$\Gamma \vd a R_\wbox x,  \Delta$}
\LL{\fns $\bdia \dashv \wbox$}
\UIC{$\Gamma \vd  x R_\bdia a,  \Delta$}
\DP 
 \ \ \ 
\AXC{$\Gamma \vd a R_\rhd b,  \Delta$}
\LL{\fns $\rhd \dashv \;\blacktriangleright$}
\UIC{$\Gamma \vd  b R_\blacktriangleright a,  \Delta$}
\DP 
 \\[\rulespace]
\AXC{$\Gamma \vd a  R_\bbox x,  \Delta$}
\RL{\fns $\wdia \dashv \bbox^{-1}$}
\UIC{$\Gamma \vd x R_\wdia a,  \Delta$}
\DP 
 \ \ \    
\AXC{$\Gamma \vd x R_\bdia a,  \Delta$}
\RL{\fns $\bdia \dashv \wbox^{-1}$}
\UIC{$\Gamma \vd a R_\wbox x,  \Delta$}
\DP
 \ \ \   
\AXC{$\Gamma \vd a R_\blacktriangleright b,  \Delta$}
\RL{\fns $\blacktriangleright \,\dashv \,\rhd$}
\UIC{$\Gamma \vd  b R_\rhd a,  \Delta$}
\DP 
 \\
\end{tabular}
\par}}

\noindent 
Adjunction rules  encode the fact that operators $\wdia$ and $\bbox$,  $\bdia$ and $\wbox$, and $\rhd$ and $\blacktriangleright$ constitute pairs of adjoint operators. 

{{\centering
\begin{tabular}{rl}
\mc{2}{c}{\rule[-1.85mm]{0mm}{8mm} \textbf{Invertible logical rules for propositional connectives}\rule[-1.85mm]{0mm}{8mm}} \\

\AXC{$\Gamma, a : A, a: B \vd \Delta$}
\LL{$ \wedge_L$}
\UIC{$\Gamma, a : A \wedge B \vd \Delta$}
\DP
 \ & \  
\AXC{$\Gamma \vd a: A, \Delta$}
\AXC{$\Gamma \vd a: B, \Delta$}
\RL{$\wedge_R$}
\BIC{$\Gamma \vd a: A \wedge B, \Delta$}
\DP
 \\[\rulespace]

\AXC{$\Gamma \vd  x:: A, \Delta$}
\AXC{$\Gamma \vd  x:: B, \Delta$}
\LL{$\vee_L$}
\BIC{$\Gamma \vd  x:: A \vee B, \Delta$}
\DP
 \ & \  
\AXC{$\Gamma, x:: A, x:: B \vd \Delta$}
\RL{$ \vee_R$}
\UIC{$\Gamma, x:: A \vee B \vd \Delta$}
\DP


%


\end{tabular}
\par}}

{{\centering
\begin{tabular}{rl}

\mc{2}{c}{\rule[-1.85mm]{0mm}{8mm}\textbf{Invertible logical rules for modal connectives${}^\ast$}\rule[-1.85mm]{0mm}{8mm}} \\

\AXC{$ \Gamma, a:\Box A  \vd x:: A, a R_\Box x, \Delta $}
\LL{$\wbox_L$}
\UIC{$\Gamma, a:\Box A \vd a R_\Box x, \Delta$}
\DP
 \ & \  
\AXC{$ \Gamma, x:: A \vd a R_\Box x, \Delta $}
\RL{$\Box_R$}
\UIC{$\Gamma \vd a: \Box A, \Delta$}
\DP
 \\[\rulespace]

\AXC{$\Gamma, a:A \vd x R_\Diamond a, \Delta $}
\LL{$\Diamond_L$}
\UIC{$\Gamma \vd x:: \Diamond A, \Delta$}
\DP
 \ & \   
\AXC{$\Gamma, x:: \Diamond A \vd  a:A, x R_\Diamond a, \Delta$}
\RL{$\Diamond_R$}
\UIC{$\Gamma,  x:: \Diamond A  \vd x R_\Diamond a, \Delta$}
\DP
 \\[\rulespace]
\AXC{$ \Gamma, a:\rhd A  \vd b: A, a R_\rhd b, \Delta $}
\LL{$\rhd_L$}
\UIC{$\Gamma, a:\rhd A \vd a R_\rhd b, \Delta$}
\DP
 \ & \  
\AXC{$ \Gamma, b: A \vd a R_\rhd b, \Delta $}
\RL{$\rhd_R$}
\UIC{$\Gamma \vd a: \rhd A, \Delta$}
\DP
 \\[\rulespace]
\mc{2}{p{0.8\textwidth}}{\centering${}^\ast$Side condition: the variable $x$ (resp.~$a$, resp.~$b$) must not occur \\ in the conclusion of $\wbox_R$  (resp.~ $\wdia_L$, resp.~$\rhd_R$).} \\
\end{tabular}
\par}}

\noindent
Logical rules encode the definition of satisfaction and refutation for propositional and modal connectives discussed in Section \ref{ssec:relsem}. The proof of their soundness in Appendix \ref{appendix:soundbase} shows how this encoding works.

\subsection{Relational calculi for the axiomatic extensions of the basic $\mathcal{L}$-logic } \label{ssec: Relational calculus extension}
The structural rule corresponding to each axiom listed in Table \ref{axiom table} is generated as the read-off of the first-order condition corresponding to the given axiom as listed in Proposition \ref{lemma:correspondences}. For any nonempty subset $\Sigma$ of modal axioms as reported in Table \ref{axiom table}, we let $\mathbf{R.L}\Sigma$ denote the extension of $\mathbf{R.L}$ with the corresponding rules.
{\footnotesize{
\begin{table}
    \centering
   \begin{tabular}{|c|Sc||c|Sc|}
\hline
   \textbf{\ Modal axiom \ } & \textbf{\ Relational calculus rule \ }&\textbf{\ Modal axiom \ } & \textbf{\ Relational calculus rule \ }\\ 
   \hline
     $\wbox p \vd p$  &  \AXC{$\Gamma \vd a R_\wbox x, \Delta$}\UIC{$\Gamma \vd a I x, \Delta$} \DP 
   &   $ p \vd \wdia p $ &  \AXC{$\Gamma \vd x R_\wdia a, \Delta$}\UIC{$\Gamma \vd a I x, \Delta$} \DP \\[2ex]
     \hline 
    $p \vd \wbox\wdia p$ &  \AXC{$\Gamma \vd x R_\wdia a, \Delta$}\UIC{$\Gamma \vd x R_\bdia a, \Delta$} \DP 
     & $ \wdia \wbox p \vd  p $ &  \AXC{$\Gamma \vd x R_\bdia a, \Delta$}\UIC{$\Gamma \vd x R_\wdia a, \Delta$} \DP \\[2ex]
     \hline
      $\wbox p \vd \wbox\wbox p$ &  \AXC{$\Gamma  \vd  a R_\wbox x, \Delta$}\UIC{$\Gamma, b R_\wbox x \Rarr  y J b  \vd a R_\wbox y, \Delta$} \DP 
    &  $ \wdia p \vd \wdia\wdia p $ &  \AXC{$\Gamma \vd x R_\wdia a, \Delta$}\UIC{$\Gamma, y R_\wdia a \Rarr b I y \vd  x R_\wdia b,  \Delta$} \DP \\[2ex]
     \hline 
      $  p \vd {\rhd}{\rhd} p $ &  \AXC{$\Gamma \vd a R_\rhd b, \Delta$}\UIC{$\Gamma \vd  b R_\rhd a,  \Delta$} \DP &&\\[2ex]
     \hline 
\end{tabular}
\vspace{0.2 cm}
    \caption{Modal axioms and their corresponding rules.}
    \label{axiom table}
\end{table}
}}
\subsection{The relational calculus $\mathbf{R.L}\rho$ for the $\mathcal{L}$-logic of rough formal contexts}
\label{ssec: relational Kent calculus}

The calculus $\mathbf{R.L}$ introduced in Section \ref{sec:calculus} can be specialized so as to capture the semantic environment of rough formal contexts by associating the connective $\kbox$ (resp.\ $\kdia$) with relational labels in which $\kentsbox$ (resp.\ $\kentsdiamond$) occurs, and adding rules encoding the reflexivity and the transitivity of $E$, rather than  the (equivalent, cf.~Lemma \ref{thm:equivalent_s_rewriting}) first-order conditions on  $\kentsbox$. We need the following set of switching rules encoding the relation between $E$ and $I$, and the $I$-compatibility of $E$ and $\kentsbox$ (and $\kentsdiamond$).
\smallskip

{{\centering
\begin{tabular}{rl}
 \mc{2}{c}{\rule[-1.85mm]{0mm}{8mm}\textbf{Interdefinability rules}\rule[-1.85mm]{0mm}{6mm}} \\
\AXC{$\Gamma, b \kentsbox x \Rightarrow b I y \vd y :: A, \Delta$}
\LL{swSf$^*$}
\UIC{$\Gamma, a : A \vd a \kentsbox x, \Delta$}
\DP
 \quad\quad &
\AXC{$\Gamma, a : A \vd a \kentsbox x, \Delta$}
\RL{swSfi$^*$}
\UIC{$\Gamma, b \kentsbox x \Rightarrow b I y \vd y :: A, \Delta$}
\DP \\[\rulespace]

\AXC{$\Gamma, x :: A \vd x \kentsdiamond a, \Delta$}
\LL{swSdf$^*$}
\UIC{$\Gamma, b E a \vd b : A, \Delta$}
\DP
\quad\quad
&
\AXC{$\Gamma, b E a \vd b : A, \Delta$}
\RL{swSdfi$^*$}
\UIC{$\Gamma, x :: A \vd x \kentsdiamond a, \Delta$}
\DP \\[\rulespace]

\AXC{$\Gamma, a E c \vd a \kentsbox x, \Delta$}
\LL{swES$^*$}
\UIC{$\Gamma, b \kentsbox x \Rightarrow b I y \vd y \kentsdiamond a, \Delta$}
\DP \quad\quad
&
\AXC{$\Gamma, b \kentsbox x \Rightarrow b I y \vd y \kentsdiamond a, \Delta$}
\RL{swESi$^*$}
\UIC{$\Gamma, a E c \vd a \kentsbox x, \Delta$}
\DP\\[\rulespace]

\AXC{$\Gamma \vd a \kentsbox x, \Delta$}
\LL{curryS$^{**}$}
\UIC{$\Gamma, b E a \vd b I x, \Delta$}
\DP \quad\quad
&
\AXC{$\Gamma, b E a \vd b I x, \Delta$}
\RL{uncurryS$^{**}$}
\UIC{$\Gamma, \vd a \kentsbox x, \Delta$}
\DP \\
[\rulespace]
\mc{2}{p{0.8\textwidth}}{\centering \small ${}^\ast$Side condition: the variables  $y,a,b$  occurring as  labels to a formula \\ in the premise of any of these rules do not occur in $\Gamma, \Delta$.} \\
\mc{2}{p{0.8\textwidth}}{\centering \small $^{\ast\ast}$Side condition: $b$ does not occur $\Gamma, \Delta$.} \\
\end{tabular}

\begin{tabular}{rcl}
 \mc{3}{c}{\rule[-1.85mm]{0mm}{8mm}\textbf{Rules for equivalence relations}\rule[-1.85mm]{0mm}{6mm}} \\
\AXC{$\Gamma,   a E a \vd  \Delta$}
\LL{refl}
\UIC{$\Gamma  \vd  \Delta$}
\DP \quad
&
\AXC{$\Gamma  \vd a E b, \Delta$}
\RL{sym}
\UIC{$\Gamma \vd b E a, \Delta$}
\DP \quad 
&
\AXC{$\Gamma  \vd a E b, b E c  \Delta$}
\RL{trans}
\UIC{$\Gamma  \vd  a E c, \Delta$}
\DP\\
\end{tabular}
\par}}

\section{Properties of $\mathbf{R.L}\rho$ and $\mathbf{R.L}\Sigma$}
\label{sec:kentstory}
\subsection{Soundness}
\label{ssec:soundness}
Any sequent $\Gamma\vd \Delta$ is to be interpreted in any enriched formal $\mathcal{L}$-context $\mathbb{F} = (\mathbb{P}, R_\Box, R_\Diamond, R_\rhd)$ based on $\mathbb{P} = (A, X, I)$ in the following way: for any assignment $V:\mathsf{Prop} \to \mathbb{P}^+$ that can be uniquely extended to an assignment on $\mathcal{L}$-formulas, and for any interpretation of labels $\alpha: \{a, b, c, \ldots \} \to A$ and $\chi: \{x, y, z, \ldots \} \to X$, we let $\iota_{(V, \alpha, \chi)}$ be the interpretation of well-formed formulas and well-formed terms indicated in the following table:
\smallskip

{{\centering
\begin{tabular}{|c@{\hskip 2mm}| c@{\hskip 2mm}| |c@{\hskip 2mm}|c@{\hskip 2mm} |}
\hline
$a :  A$ & $\alpha(a)\in \val{A}_V$ &
$x :: A$ & $\chi(x)  \in \descr{A}_V$ \\
\hline
$a R_\Box x$ & $\alpha(a)R_\Box \chi(x)$ &
$a R_\blacksquare x$ & $\alpha(a)R_\blacksquare \chi(x)$ \\
\hline
$x R_\Diamond a$ & $\chi(x)R_\Diamond \alpha(a)$ &
$x R_\Diamondblack a$ & $\chi(x)R_\Diamondblack \alpha(a)$ \\
\hline
$a R_\rhd b$ & $\alpha(a)R_\Diamond \alpha(b)$ &
$a R_\blacktriangleright b$ & $\alpha(a)R_\Diamondblack \alpha(b)$ \\
\hline
$aIx$ & $\alpha(a)I\chi(x)$ &
$t_1(u) \Rightarrow t_2(u)$ & $\forall u(\iota_{(V, \alpha, \chi)}(t_1(u)) \Rightarrow \iota_{(V, \alpha, \chi)}(t_2(u))) $  \\
\hline

\end{tabular}
\par}}
\smallskip

\noindent
Under this interpretation, sequents   $\Gamma\vdash \Delta$ are interpreted as follows\footnote{The symbols $\with$ and $\parr$ denotes a meta-linguistic conjunction and a disjunction, respectively.}:  \[\forall V \forall \alpha \forall \chi(\bigwith_{\gamma \in \Gamma} \iota_{(V, \alpha, \chi)}(\gamma)   \Longrightarrow  \bigparr_{\delta \in \Delta} \iota_{(V, \alpha, \chi)}(\delta)).\]

In the following, we show the soundness of the interdefinability rules in \textbf{R.L}$\rho$, being the proof of soundness of the (pure structure) switch rules similar. The soundness of the rules for the basic calculus \text{R.L} is proved in Appendix \ref{appendix:soundbase}.

\begin{remark}
\label{remark:technical_approxm}
Given a polarity $\mathbb{P} = (A, X, I)$, $c \in \mathbb{P}^+$, and $B\subseteq A$, the condition
\smallskip

{{\centering 
$(\forall x \in X)(c \subseteq I^{(0)}[x] \Rightarrow B \subseteq I^{(0)}[x])$,
\par}}
\smallskip

\noindent can be rewritten using the defining properties of $\bigcap$ as the inclusion
\smallskip

{{\centering 
$B \subseteq \bigcap \left\{ I^{(0)}[x] \mid x \in X, c \subseteq I^{(0)}[x] \right\}$,
\par}}
\smallskip

\noindent which, by Lemma \ref{lemma:s_as_i_e}, is equivalent to $B \subseteq c$.
\end{remark}

\begin{lemma}
\label{lemma:soundnessextrarules}
The rules swSf, swSfi, swSdf, swSdfi, swES, swESi, curryS, uncurryS, refl, sym, and trans are sound with respect to the class of rough formal contexts.
\end{lemma}
\begin{proof}
Under the assumption that $E$ and $\kentsbox$ are $I$-compatible, all the formulae are interpreted as concepts. In what follows, we will refer to the objects (resp.\ features) occurring in $\Gamma$ and $\Delta$ in the various rules with $\overline d$ (resp.\ $\overline w$). For the sake of readability, in what follows we omit an explicit reference to the interpretation maps $\alpha$ and $\chi$.

\noindent (swSf and swSfi) 
\smallskip

{{\small\centering 
\begin{tabular}{rll}
     & $\forall V\forall \overline d \forall \overline w\forall x \forall y\left(\bigwith \Gamma \with \forall b(b \kentsbox x \Rightarrow bIy) \Rightarrow y \in \descr{A}_V \parr \bigparr \Delta  \right)$ \\
     iff &  $\forall V\forall \overline d \forall \overline w\forall x \forall y\left(\bigwith \Gamma \with \kentsbox^{(0)}[x] \subseteq I^{(0)}[y] \Rightarrow y \in \descr{A}_V \parr \bigparr \Delta  \right)$ & \ \  Def.~of $(\cdot)^{(0)}$ \ \\
     iff &  $\forall V\forall \overline d \forall \overline w\forall x \forall y\left(\bigwith \Gamma \with \kentsbox^{(0)}[x] \subseteq I^{(0)}[y] \Rightarrow \val{A}_V \subseteq I^{(0)}[y] \parr \bigparr \Delta  \right)$  & \ \ $V(A)$ closed\\
     iff &  $\forall V\forall \overline d \forall \overline w\forall x \left(\bigwith \Gamma \Rightarrow  \forall y\left(\kentsbox^{(0)}[x] \subseteq I^{(0)}[y] \Rightarrow \val{A}_V \subseteq I^{(0)}[y]\right) \parr \bigparr \Delta  \right)$  & \ \ uncurrying + side\\
     iff &  $\forall V\forall \overline d \forall \overline w\forall x \left(\bigwith \Gamma \Rightarrow  \val{A}_V \subseteq \kentsbox^{(0)}[x] \parr \bigparr \Delta  \right)$  & \ \ $S$ $I$-comp, Remark \ref{remark:technical_approxm}\\
     iff &  $\forall V\forall \overline d \forall \overline w\forall x \left(\bigwith \Gamma \Rightarrow  \forall a\left(a \in \val{A}_V \Rightarrow a\in \kentsbox^{(0)}[x]\right) \parr \bigparr \Delta  \right)$  & \ \ Def.~of $\subseteq$\\
     iff &  $\forall V\forall \overline d \forall \overline w\forall x \left(\bigwith \Gamma \with a \in \val{A}_V \Rightarrow a\in \kentsbox^{(0)}[x] \parr \bigparr \Delta  \right)$  & \ \ currying\\
     iff &  $\forall V\forall \overline d \forall \overline w\forall x \left(\bigwith \Gamma \with a \in \val{A}_V \Rightarrow a  \kentsbox x \parr \bigparr \Delta  \right)$  & \ \ Def.~of $(\cdot)^{(0)}$\\
\end{tabular}
\par}}
\smallskip

\noindent (swSdf and swSdfi)
\smallskip

{{\small\centering 
\begin{tabular}{rll}
& $\forall V\forall \overline d \forall \overline w\forall a \forall x\left(\bigwith \Gamma \with x \in \descr{A}_V \Rightarrow x \kentsdiamond a \parr \bigparr \Delta  \right)$ \\
iff & $\forall V\forall \overline d \forall \overline w\forall a \forall x\left(\bigwith \Gamma \with \val{A}_V \subseteq I^{(0)}[x] \Rightarrow x \in \kentsdiamond^{(0)}[a] \parr \bigparr \Delta  \right)$ & \ \ $V(A)$ closed \\
iff & $\forall V\forall \overline d \forall \overline w\forall a \forall x\left(\bigwith \Gamma \with \val{A}_V \subseteq I^{(0)}[x] \Rightarrow I^{(0)}[\kentsdiamond^{(0)}[a]] \subseteq I^{(0)}[x] \parr \bigparr \Delta  \right)$ & \ \  $S$ is $I$-compatible \\
iff & \multicolumn{2}{l}{\begin{tabular*}{\textwidth}{@{}l@{\extracolsep{\fill}}r@{}}
     $\forall V\forall \overline d \forall \overline w\forall a \left(\bigwith \Gamma \Rightarrow I^{(0)}[\kentsdiamond^{(0)}[a]] \subseteq \val{A}_V \parr \bigparr \Delta  \right)$ & \ \ \ $V(A)$ closed, Remark \ref{remark:technical_approxm} 
\end{tabular*}}\\
iff & $\forall V\forall \overline d \forall \overline w\forall a\forall b \left(\bigwith \Gamma \with b \in I^{(0)}[\kentsdiamond^{(0)}[a]] \Rightarrow b \in \val{A}_V \parr \bigparr \Delta  \right)$ & \ \ Def.\ of $\subseteq$ \\
iff & $\forall V\forall \overline d \forall \overline w\forall a\forall b \left(\bigwith \Gamma \with b \in I^{(0)}[J^{(0)}[E^{(0)}[a]]] \Rightarrow b \in \val{A}_V \parr \bigparr \Delta  \right)$ & \ \ Remark \ref{lemma:s_as_i_e}\\
iff & $\forall V\forall \overline d \forall \overline w\forall a\forall b \left(\bigwith \Gamma \with b \in I^{(0)}[I^{(1)}[E^{(0)}[a]]] \Rightarrow b \in \val{A}_V \parr \bigparr \Delta  \right)$ & \ \ Def.~of $J$\\
iff & $\forall V\forall \overline d \forall \overline w\forall a\forall b \left(\bigwith \Gamma \with b \in E^{(0)}[a] \Rightarrow b \in \val{A}_V \parr \bigparr \Delta  \right)$ & \ \ $E$ is $I$-compatible \\
iff & $\forall V\forall \overline d \forall \overline w\forall a\forall b \left(\bigwith \Gamma \with bEa \Rightarrow b \in \val{A}_V \parr \bigparr \Delta  \right)$ & \ \ Def.~of $(\cdot)^{(0)}$\\
\end{tabular}
\par}}
\smallskip

\noindent (curryS and uncurryS)
\smallskip

{{\small\centering 
\begin{tabular}{rll}
& $\forall V\forall \overline d \forall \overline w\forall a \forall x\left(\bigwith \Gamma \Rightarrow a \kentsbox x \parr \bigparr \Delta  \right)$ \\
iff & $\forall V\forall \overline d \forall \overline w\forall a \forall x\left(\bigwith \Gamma \Rightarrow \forall b(bEa \Rightarrow b I x) \parr \bigparr \Delta  \right)$ & \ \ \ Def.~of $\kentsbox$\\
iff & $\forall V\forall \overline d \forall \overline w\forall a \forall x\forall b\left(\bigwith \Gamma \Rightarrow (bEa \Rightarrow b I x \parr \bigparr \Delta  )\right)$ & \ \ \  side condition\\
iff & $\forall V\forall \overline d \forall \overline w\forall a \forall x\forall b\left(\bigwith \Gamma \with bEa \Rightarrow b I x \parr \bigparr \Delta  \right)$ & \ \ \ currying\\
\end{tabular}
\par}}
\smallskip

\noindent (swES and swESi) The proof is similar to the previous ones.
The soundness of rules refl, sym, and trans follows from the fact that relation $E$ is equivalence relation in a rough formal context. 
\end{proof}

\begin{remark}
\label{remark:switchicomp}
The soundness of the switch rules is proved exactly as the soundness of the interdefinability rules in Lemma \ref{lemma:soundnessextrarules} by the $I$-compatibility of the relations in enriched formal contexts. More in general, these rules encode {\em exactly} the $I$-compatibility of such relations. Let us show this for $R_\Box$, as the others are proved similarly. One of the two $I$-compatibility conditions can be rewritten as
\smallskip

{\small
{{\centering
\begin{tabular}{rlr}
& $I^{(0)}[I^{(1)}[R_\Box^{(0)}[x]]] \subseteq R_\Box^{(0)}[x]$ & \\
iff & $\forall y(y \in I^{(1)}[R_\Box^{(0)}[x]] \Rightarrow aIy)\Rightarrow aR_\Box x$ & Def.\ of $I^{(0)}[\cdot]$\\
iff & $\forall y(\forall b(b R_\Box x \Rightarrow bIy) \Rightarrow aIy)\Rightarrow aR_\Box x$ & Def.\ of $I^{(1)}[\cdot]$\\
\end{tabular}
\par}}
 }

\smallskip

\noindent In what follows we are not assuming that $R_\Box$ is $I$-compatible; hence the valuation of an arbitrary formula does not need to be closed, but rather just a pair containing an arbitrary set of objects and its intension, or a an arbitrary set of features and its extension.  Ignoring the contexts for readability, the rule $S_{x\Box a}$ is interpreted as
\smallskip

{\small
{{\centering
\begin{tabular}{rlr}
& $\forall V, a, x \left(\forall y\left( \forall b(b R_\Box x \Rightarrow bIy) \Rightarrow y \in \descr{A}_V \right) \Longrightarrow (a \in \val{A}_V \Rightarrow a R_\Box x) \right)$ & \\
iff & $\forall V, a, x \left(\forall y\left( \forall b(b R_\Box x \Rightarrow bIy) \Rightarrow y \in \descr{A}_V \right) \Longrightarrow (\val{A}_V \subseteq R_\Box^{(0)}[x])\right)$ & Def.\ of $R_\Box^{(0)}[\cdot]$\\
iff & $\forall V, a, x \left(\forall y\left( y \in I^{(1)}[R_\Box^{(0)}[x]]  \Rightarrow y \in \descr{A}_V \right) \Longrightarrow (\val{A}_V \subseteq R_\Box^{(0)}[x])\right)$ & Def.\ of $I^{(1)}[\cdot]$\\
iff & $\forall V, a, x \left(\forall y\left( y \in I^{(1)}[R_\Box^{(0)}[x]]  \Rightarrow y \in \descr{A}_V \right) \Longrightarrow (\val{A}_V \subseteq R_\Box^{(0)}[x])\right)$ & Def.\ of $I^{(1)}[\cdot]$\\
implies & $\forall V, a, x \left(\val{A}_V \subseteq I^{(0)}[I^{(1)}[R_\Box^{(0)}[x]]] \Longrightarrow (\val{A}_V \subseteq R_\Box^{(0)}[x])\right)$ & $I^{(0)}[\cdot])$ antitone\footnote{And also $\val{A}_V\subseteq I^{0}[\descr{A}_V]$ holds in both the cases: the one where $\val{A}$ is the extension of an arbitrary set of features, and when $\descr{A}$ is the intension of $\val{A}$.  }\\
iff & $\forall V, a, x \left(I^{(0)}[I^{(1)}[R_\Box^{(0)}[x]]] \subseteq R_\Box^{(0)}[x]\right)$ & $I^{(0)}[\cdot])$ Def.\ of $\subseteq$
\end{tabular}
\par}}
}
\smallskip 

\noindent The second $I$-compatibility condition for $R_\Box$ is proved similarly using $S_{a\Box x}$.
\end{remark}

\subsection{Syntactic completeness of the basic calculus and its axiomatic extensions}

In the present section, we show that the axioms and rules of $\mathbf{R.L}\Sigma$, where $\Sigma$ is a subset of the set of axioms in Table \ref{axiom table}, are derivable in $\mathbf{R.L}$ extended with the corresponding rules. The axioms and rules of the basic logic $\mathbf{L}$ and some of its axiomatic extensions are discussed in Appendix \ref{appendix_completeness}.
Below, we show how the  axioms $\kbox p \vd p$, $\kbox p \vd \kbox \kbox p$, and  $ p \vd {\rhd}{\rhd} p$ can be derived using rules refl, sym, and trans respectively.

{\small
\begin{center}
\begin{tabular}{@{}c@{}c@{}c@{}}
\bottomAlignProof
\AXC{$x:: p \vd x:: p$}
\LL{$\Box_L$}
\UIC{$b : \kbox p, x :: p \vd b\kentsbox  x$}
\LL{curry}
\UIC{$b : \kbox p, x :: p, bEb\vd bIx$}
\LL{refl}
\UIC{$b : \kbox p, x :: p \vd bIx$}
\LL{approx$_x$}
\UIC{$b : \kbox p \vd b : p $}
\DP
 & 
\bottomAlignProof
\AXC{$x :: p \vd x :: p$}
\RL{$\Box_L$}
\UIC{$a:\Box p, x :: p \vd a \kentsbox x$}
\RL{curry}
\UIC{$a : \Box p, x :: p, bEa \vd bIx$}
\RL{trans}
\UIC{$a : \Box p, x :: p, bEc, c E a \vd bIx$}
\RL{uncurry}
\UIC{$a : \Box p, x::p, c E a \vd c \kentsbox x$}
\RL{$\Box_R$}
\UIC{$a: \kbox p, c E a \vd c: \Box p$}
\RL{swSdfi}
\UIC{$a : \kbox p, y :: \kbox p, \vd y \kentsdiamond a$}
\RL{$\Box_R$}
\UIC{$a: \kbox p \vd a: \kbox\kbox p$}
\DP
 & 
\bottomAlignProof
\AXC{$b : {\rhd} p, a:p \vd  a:p, b R_\rhd a$}
\RL{$\rhd_L$}
\UIC{$b : {\rhd} p, a:p \vd b R_\rhd a$}
\RL{sym}
\UIC{$b : {\rhd} p, a:p \vd a R_\rhd b$}
\RL{$\rhd_R$}
\UIC{$a: p \vd a: {\rhd}{\rhd}p$}
\DP
 \\
\end{tabular}
\end{center}
}
\section{Conclusions} \label{sec:Conclusions}
In the present paper, we have introduced labelled calculi for a finite set of non-distributive modal logics in a modular way, and we have shown that the calculus associated with each such logic is sound w.r.t.~the relational semantics of that logic given by elementary classes of enriched formal contexts, and syntactically complete w.r.t.~the given logic. These results showcase that  the methodology introduced in \cite{negri2005proof} for introducing labelled calculi by suitably integrating semantic information in the design of the rules can be extended from classical modal logics to the wider class of non-distributive logics. This methodology has proved successful for designing calculi for classical modal logics enjoying excellent computational properties, such as cut elimination, subformula property, being contraction-free, and being suitable for proof-search. Future developments of this work include the proofs of these results for the calculi introduced in the present paper.  
\begin{appendix}
\section{Soundness of the basic calculus}
\label{appendix:soundbase}
\begin{lemma}
The basic calculus $\mathbf{R.L}$ is sound  for the logic of enriched formal contexts.
\end{lemma}
\begin{proof}
The soundness of the axioms, cut rules and propositional rules is trivial from the definitions of satisfaction and refutation relation for enriched formal contexts. 
We now discuss the soundness for the other rules.

\noindent \textbf{Adjunction rules}. The soundness of the adjunction rules follows from the fact that $R_\bbox = R_\wdia^{-1}$, $R_\bdia = R_\wbox^{-1}$ and $R_\rhd = R_\blacktriangleright^{-1}$. 

\noindent \textbf{Approximation rules}. We only give proof for $approx_a$. The proof for $approx_x$ is similar. In what follows, we will refer to the objects (resp.\ features) occurring in $\Gamma$ and $\Delta$ in the various rules with $\overline d$ (resp.\ $\overline w$).

{\small
\begin{center}
\begin{tabular}{rll}
     & $\forall V\forall \overline d \forall \overline w\forall a \forall x\left(\bigwith \Gamma \with x \succ A  \Rightarrow a I x \parr \bigparr \Delta  \right)$ \\
      iff  & $\forall V\forall \overline d \forall \overline w\forall a \left(\bigwith \Gamma \with \forall x (x \succ A  \Rightarrow a I x) \parr \bigparr \Delta  \right)$  & \ \ \ $x$ does not appear in $\Gamma$ or $\Delta$ \\
     iff &  $\forall V\forall \overline d \forall \overline w\forall a \forall x\left(\bigwith \Gamma \with x \in  \descr{V(A)}  \Rightarrow a I x \parr \bigparr \Delta  \right)$ \\
      iff &  $\forall V\forall \overline d \forall \overline w\forall a \forall x\left(\bigwith \Gamma \with a \in  I^{(0)}\descr{V(A)}  \parr \bigparr \Delta  \right)$ & \ \ \ Def. of $(\cdot)^{(0)}$\\
      iff &  $\forall V\forall \overline d \forall \overline w\forall a \forall x\left(\bigwith \Gamma \with a \in \val{V(A)}  \parr \bigparr \Delta  \right)$ & \ \ \ $V(A)$ is closed\\
      iff &  $\forall V\forall \overline d \forall \overline w\forall a \forall x\left(\bigwith \Gamma \with a  \Vdash A  \parr \bigparr \Delta  \right)$\
\end{tabular}
\end{center}
}

\noindent \textbf{Invertible  rules for modal connectives}. We only give proofs for $\wbox_L$ and $\wbox_R$. The proofs for  $\wdia_R$, $\wdia_L$, $\rhd_R$, and $\rhd_L$ can be given in a similar manner.

{\small
\begin{center} 
\begin{tabular}{rll}
     & $\forall V\forall \overline d \forall \overline w\forall x \forall y\left(\bigwith \Gamma \with a \Vdash \wbox A \Rightarrow x \succ A \parr a R_\wbox x \parr \bigparr \Delta  \right)$ \\
     implies &  $\forall V\forall \overline d \forall \overline w\forall x \forall y\left(\bigwith \Gamma \with a \Vdash \wbox A \Rightarrow \forall b (b \vdash \wbox A \Rightarrow b R_\wbox x) \parr a R_\wbox x \parr \bigparr \Delta  \right)$ & \ \ \   Def. of $\wbox$  \\ 
     implies &  $\forall V\forall \overline d \forall \overline w\forall x \forall y\left(\bigwith \Gamma \with a \Vdash \wbox A \Rightarrow  a R_\wbox x \parr \bigparr \Delta  \right)$   \\ 
\end{tabular}
\end{center}
 }

The invertibility of the rule $\wbox_L$  is obvious from the fact that the premise can be obtained from the conclusion by weakening.

{\small
\begin{center}
\begin{tabular}{rlr}
& $\forall V\forall \overline d \forall \overline w\forall a \forall x\left(\bigwith \Gamma \with x \succ A \Rightarrow a R_\wbox x \parr \bigparr \Delta  \right)$ \\
iff & $\forall V\forall \overline d \forall \overline w\forall a\left(\bigwith \Gamma \with  \forall x(x \succ A \Rightarrow a R_\wbox x) \parr \bigparr \Delta  \right)$  & \quad $x$ does not appear in $\Gamma$ or $\Delta$\\
iff & $\forall V\forall \overline d \forall \overline w\forall a\left(\bigwith \Gamma \with  \Rightarrow a \Vdash \wbox A \parr \bigparr \Delta  \right)$  & $x$ Def. of $\wbox$\\
\end{tabular}
\end{center}
 }

\noindent \textbf{Switch rules.} Soundness of the rules S$xa$ and S$ax$ follows from the fact that for any concepts $c_1$ and $c_2$ we have 
\begin{align*} 
\val{c_1} \subseteq \val{c_2} \quad \iff \quad \descr{c_2} \subseteq \descr{c_1}. 
\end{align*}
The soundness of all other switch rules follows from the definition of modal connectives and I-compatibility. As all the proofs are similar we only prove the soundness of S$a \wdia x$ as a representative case. Soundness of other rules can be proved in an analogous manner. 

{\small
\begin{center}
\begin{tabular}{rll}
     & $\forall V\forall \overline d \forall \overline w\forall a \forall b \left(\bigwith \Gamma \with \forall y (y R_\wdia a \Rightarrow b I y) \Rightarrow b \Vdash A \parr \bigparr \Delta  \right)$ \\
      iff & $\forall V\forall \overline d \forall \overline w\forall a \forall b\left(\bigwith \Gamma \with  b \in I^{(0)}[R_\wdia^{(0)}[a]]  \Rightarrow b \Vdash A \parr \bigparr \Delta  \right)$& \ \ \  Def. of $R_\wdia^{(0)}$ and $I^{(0)}$ \\
      iff & $\forall V\forall \overline d \forall \overline w\forall a\left(\bigwith \Gamma \Rightarrow   \forall b (b \in I^{(0)}[R_\wdia^{(0)}[a]]  \Rightarrow b \Vdash A)  \parr \bigparr \Delta  \right)$& \ \ \  $b$ does not appear in $\Gamma$ or $\Delta$   \\
    iff & $\forall V\forall \overline d \forall \overline w\forall a\left(\bigwith \Gamma \Rightarrow   I^{(0)}[R_\wdia^{(0)}[a]]  \subseteq \val{V(A)}\parr  \bigparr \Delta  \right)$& \ \ \  $b$ does not appear in $\Gamma$ or $\Delta$   \\
    iff & $\forall V\forall \overline d \forall \overline w\forall a\left(\bigwith \Gamma  \Rightarrow  I^{(1)}[\val{V(A)}] \subseteq I^{(1)}[ I^{(0)}[R_\wdia^{(0)}[a]]] \parr \bigparr \Delta  \right)$& \ \ \  $I^{(1)}$ is antitone and  $\val{V(A)}$ is closed \\
    iff & $\forall V\forall \overline d \forall \overline w\forall a\left(\bigwith \Gamma \Rightarrow I^{(1)}[\val{V(A)}] \subseteq R_\wdia^{(0)}[a] \parr  \bigparr \Delta  \right)$& \ \ \  $R_\wbox$ is I-compatible\\
      iff & $\forall V\forall \overline d \forall \overline w\forall a\left(\bigwith \Gamma \Rightarrow  \forall x (x \in I^{(1)}[\val{V(A)}] \Rightarrow x \in R_\wdia^{(0)}[a])  \parr \bigparr \Delta  \right)$\\
      implies  & $\forall V\forall \overline d \forall \overline w\forall a\forall x \left(\bigwith \Gamma \with  x \in I^{(1)}[\val{V(A)}] \Rightarrow x \in R_\wdia^{(0)}[a])  \parr \bigparr \Delta  \right)$\\
      iff &  $\forall V\forall \overline d \forall \overline w\forall a\forall x \left(\bigwith \Gamma \with  x\succ A  \Rightarrow x  R_\wdia a \parr \bigparr \Delta  \right)$&\ \ \  Def.~of $R_\wdia^{(0)}$\\
      \smallskip
\end{tabular}
\end{center}
}

Soundness of the axiomatic extensions considered in Section \ref{ssec: Relational calculus extension} is immediate from the Proposition \ref{lemma:correspondences}. 
\end{proof}

\section{Syntactic completeness} \label{appendix_completeness}
 
As to the axioms and rules of the basic logic $\mathbf{L}$, below, we only derive in $\mathbf{R.L}$ the axioms and rules encoding the fact that $\wdia$ is a normal modal operator plus the axiom $p \vd p \vee q$.   
{\small
\begin{center}
\begin{tabular}{cc}
\AXC{ \ }
\LL{\fns Id$_{b:A}$}
\UIC{$x:: \wdia A, x:: \wdia B, b: A \vd b: A, x\RWD b$}
\LL{$\wdia_R$}
\UIC{$x:: \wdia A, x:: \wdia B, b: A \vd x\RWD b$}
\LL{$\vee_R$}
\UIC{$x:: \wdia A \vee \wdia B, b: A \vd x\RWD b$}
\LL{\fns $\wdia \dashv \bbox$}
\UIC{$x:: \wdia A \vee \wdia B, b: A \vd b\RBB x$}
\LL{\fns S$ x\bbox a^{c} $ }
\UIC{$x:: \wdia A \vee \wdia B,  a\RBB x \Rarr aIy\vd y: A$}
\AXC{ \ }
\RL{\fns Id$_{b:B}$}
\UIC{$x:: \wdia A, x:: \wdia B, b: B \vd b: B, x\RWD b$}
\RL{$\wdia_R$}
\UIC{$x:: \wdia A, x:: \wdia B, b: B \vd x\RWD b$}
\RL{$\vee_R$}
\UIC{$x:: \wdia A \vee \wdia B, b: B \vd x\RWD b$}
\RL{\fns $\wdia \dashv \bbox$}
\UIC{$x:: \wdia A \vee \wdia B, b: B \vd b\RBB x$}
\RL{\fns S$ x\bbox a^{c} $ }
\UIC{$x:: \wdia A \vee \wdia B, a\RBB x \Rarr aIy\vd y: B$}
\RL{$\vee_L$}
\BIC{$x:: \wdia A \vee \wdia B,  a\RBB x \Rarr aIy \vd y: A \vee B$}
\RL{\fns S$ x\bbox a $}
\UIC{$x:: \wdia A \vee \wdia B, a: A \vee B \vd a\RBB x$}
\RL{\fns \fns $\wdia \dashv \bbox^{-1}$}
\UIC{$x:: \wdia A \vee \wdia B, a: A \vee B \vd x\RWD a$}
\RL{$\wdia_L$}
\UIC{$x:: \wdia A \vee \wdia B \vd x:: \wdia (A \vee B)$}
\DP
\end{tabular}
\end{center}
 }

{\small
\begin{center}
\begin{tabular}{ccc}
\bottomAlignProof
\AXC{$ $}
\RL{$\bot$}
\UIC{$b R_\wbox x \Rarr b I y,  x:\bot \vd y::\bot$}
\RL{S$ x\bbox a $}
\UIC{$a:\bot, x:\bot \vd a R_\bbox x$}
\RL{\fns $\wdia \dashv \bbox^{-1}$}
\UIC{$a:\bot, x:\bot \vd x R_\wdia a$}
\RL{$\wdia_L$}
\UIC{$x:\bot \vd x: \wdia \bot$}
\DP
 &
\bottomAlignProof
\AXC{$y:: \phi \vd y::\psi$}
\RL{W}
\UIC{$y:: \phi, x::\wdia \phi \vd y::\psi, x R_\wdia a$}
\RL{S$xa$}
\UIC{$a: \psi, x::\wdia \phi  \vd a:\phi, x R_\wdia a $}
\RL{$\wdia_R$}
\UIC{$a: \psi, x::\wdia \phi \vd x R_\wdia a$}
\RL{$\wdia_L$}
\UIC{$x::\wdia \phi \vd x::\wdia \psi$}
\DP
 & 
\bottomAlignProof
\AXC{$x:p, x:q \vd x:p $}
\RL{$\vee_R$}
\UIC{$x:: p \vee q \vd x:p$}
\DP
 \\
\end{tabular}
\end{center}
}

The syntactic completeness for the other axioms and rules of $\mathbf{L}$ can be shown in a similar way. In particular, the admissibility of the substitution rule can be proved by induction in a standard manner.

We now consider the reflexivity axiom $ p \vd \wdia p$ and the transitivity axiom $\wbox p \vd \wbox\wbox p$. The derivation for dual axioms $\wbox p  \vd p$ and $\wdia\wdia p  \vd \wdia p$ can be provided analogously.

{\small
\begin{center}
\begin{tabular}{cc}
\bottomAlignProof
\AXC{$ $}
\LL{Id$_{\,a:p}$}
\UIC{$x:: p, a: p \vd  a:p, a R_\wdia x$}
\LL{$ \wdia_R$}
\UIC{$x:: p, a: p \vd a R_\wdia x$}
\LL{refl}
\UIC{$ x:\wdia p,  a: p \vd a I x $}
\LL{$approx_a$}
\UIC{$ x:\wdia p \vd x:  p$}
\DP 
 & 
\bottomAlignProof
\AXC{$ $}
\LL{Id$_{\,x::p}$}
\UIC{$a:\wbox p, x::p \vd  x::p, a R_\wbox x$}
\LL{\fns $\wbox_L$}
\UIC{$ a:\wbox p, x::p \vd  a R_\wbox x$}
\LL{trans}
\UIC{$ b R_\wbox x \Rarr z J b,  a:\wbox p, x::p \vd  a R_\wbox z$}
\dashedLine
\UIC{$ z (J;R_\wbox) x,  a:\wbox p, x::p \vd  a R_\wbox z$}
\LL{\fns $\wbox \dashv \bdia^{-1}$}
\UIC{$ z (J;R_\wbox) x,  a:\wbox p, x::p \vd  z R_\bdia a$}
\LL{Id$(J;I)_R$}
\UIC{$z (J;R_\wbox) x,  a:\wbox p, x::p \vd  z (J;(I;R_\bdia))a$}
\LL{-S$(J;S)^{ \ast}$} 
\UIC{$ b (I;R_\bdia) a,  a:\wbox p, x::p \vd  b R_\wbox x$}
\dashedLine
\UIC{$ y R_\bdia a \Rarr b I y,  a:\wbox p, x::p \vd  b R_\wbox x$}
\LL{\fns $\wbox_R$}
\UIC{$ y R_\bdia a \Rarr b I y,  a:\wbox p \vd  b:\wbox p$}
\LL{S$a \bdia x$}
\UIC{$x::\wbox p, a:\wbox p \vd x R_\bdia  a$}
\LL{\fns $\bdia \dashv \wbox^{-1}$}
\UIC{$x::\wbox p, a:\wbox p \vd a R_\wbox x$}
\LL{\fns $\wbox_R$}
\UIC{$ a:\wbox p \vd a:\wbox\wbox p$}
\DP \\
 \\
\end{tabular}
\end{center}
 }

Completeness for the  other axiomatic extensions can be shown in a similar way.

\bibliography{ref}
\bibliographystyle{plain}

\end{appendix}

\end{document}